\documentclass[12pt]{amsart}

\usepackage[a4paper,margin=2.5cm]{geometry}

\usepackage[T1]{fontenc}
\usepackage[utf8]{inputenc}
\usepackage[english]{babel}

\usepackage{amsmath,amssymb,amsthm,amsfonts}
\usepackage{mathrsfs}

\usepackage{enumitem}
\usepackage{hyperref}

\newtheorem{theorem}{Theorem}
\newtheorem{lemma}{Lemma}
\newtheorem{proposition}{Proposition}

\newtheorem{remark}{Remark}

\DeclareMathOperator{\Der}{Der}
\DeclareMathOperator{\End}{End}

\DeclareMathOperator{\codim}{codim}
\DeclareMathOperator{\Ker}{Ker}

\DeclareMathOperator{\ad}{ad}

\begin{document}
	
	\title{Lie Algebras of Locally Finite Derivations over Fields of Positive Characteristic}
		\author{Oksana Bezushchak, Anatoliy Petravchuk}
		\address{Oksana Bezushchak: Faculty of Mechanics and Mathematics,
			Taras Shevchenko National University of Kyiv, 60, Volodymyrska street, 01033  Kyiv, Ukraine}
		\email{bezushchak@knu.ua}
		\address{Anatoliy Petravchuk: Faculty of Mechanics and Mathematics,
			Taras Shevchenko National University of Kyiv, 60, Volodymyrska street, 01033  Kyiv, Ukraine}
		\email{petravchuk@knu.ua}
	\date{}
	
	\keywords{Lie algebra of derivations, locally finite derivation, locally nilpotent derivation}
	\subjclass[2020]{Primary 16W25, 17B66; Secondary 13N15, 16W20, 93C10}
	
	\maketitle
	
	\begin{abstract}
		Let \(\mathbb{F}\) be a field of positive characteristic, and let \(A\) be a finitely
	generated commutative associative \(\mathbb{F}\)-algebra. We study Lie subalgebras of
	\(\Der A\) generated by locally finite derivations.  We prove that if all commutators of the generators satisfy a natural local finiteness condition, then the associative subalgebra of $\End_{\mathbb{F}} A$ generated by these derivations is finite-dimensional.
	\end{abstract}
	
	\section{Introduction}
	
	Let \(\mathbb{F}\) be a field, and let \(A\) be a commutative
	associative \(\mathbb{F}\)-algebra generated by elements
	\(a_1,\ldots,a_m\).
	Denote by \(\Der A\) the Lie algebra of all
	\(\mathbb{F}\)-derivations of \(A\).

	Recall that a derivation	$d\in\Der A$	is called \textbf{locally nilpotent} if for every element \(a\in A\) there
	exists an integer \(n(a)\ge 1\) such that	$d^{\,n(a)}a=0.$ 	A Lie subalgebra	$L\subseteq \Der A$	is said to \textbf{act locally nilpotently} on \(A\) if every element of \(L\) is a
	locally nilpotent derivation.

	In \cite{BezushchakPetravchukZelmanov}, it was proved that every finitely generated Lie subalgebra \(L\subseteq \Der A\) acting locally nilpotently on \(A\) is finite-dimensional. M.~A.~Belabbas \cite{Belabbas}  extended this result to locally finite actions over fields of characteristic zero.
	
	The aim of this note is to establish an analogue	of this result over fields of positive characteristic.
	
A derivation \(d\in\Der A\) is said to
\textbf{act finitely} on an element \(a\in A\) if
\[
\dim_{\mathbb{F}}
\left(\sum_{k=0}^{\infty}\mathbb{F}d^k a\right)
<\infty.
\]
A derivation \(d\) is called \textbf{locally finite}
if it acts finitely on every element of \(A\).
A Lie subalgebra \(L\subseteq\Der A\) is said to
\textbf{act locally finitely} on \(A\) if every
element of \(L\) is a locally finite derivation.

	Let	$d_1,\ldots,d_n\in \Der A$,	and let \(L\) be the Lie subalgebra of \(\Der A\) generated by
	\(d_1,\ldots,d_n\). Denote by \(U\) the associative subalgebra of	$\End_{\mathbb{F}} A$	generated by these derivations.
	
	By a \textbf{commutator} in \(d_1,\ldots,d_n\),
	we mean either one of the generators
	\(d_1,\ldots,d_n\) or any element obtained from
	them by iterated Lie brackets.
	
	Our main result is the following.

	\begin{theorem}\label{thm:main} Let \(A\) be a commutative associative algebra
		over a field of positive characteristic,
		generated by elements \(a_1,\ldots,a_m\). 	Let	$d_1,\ldots,d_n\in \Der A$.	If every commutator \(\rho\) in
		\(d_1,\ldots,d_n\) acts finitely on each 		generator \(a_i\),  \(i=1,\ldots,m\), then the associative  subalgebra $U$ of \(\End_{\mathbb F}A\)  generated by $d_1,\ldots,d_n$ is finite-dimensional.	\end{theorem}
	
	\section{Auxiliary results}
	
	We first record a simple consequence of local finiteness in positive
	characteristic.
	
	\begin{lemma}\label{lem:polynomial}
		Let \(\rho\) be a commutator in \(d_1,\ldots,d_n\). Put
		\[
		S_i=
		\dim_\mathbb{F}\sum_{k=0}^{\infty}\mathbb{F}\rho^k a_i,
		\qquad
		1\le i\le m, \qquad 		\text{and} \qquad	n(\rho)=S_1+\cdots+S_m.
		\]
		Then there exists a nonzero polynomial of the form
		\[
		f_\rho(t)
		=
		\alpha_0t+\alpha_1t^p+\alpha_2t^{p^2}
		+\cdots+\alpha_{n(\rho)}t^{p^{n(\rho)}} \qquad	\text{such that} \qquad 	f_\rho(\rho)=0.	\]	Moreover, one may assume that \(f_\rho(t)\) is monic. 
	\end{lemma}
	
	\begin{proof}
		Since	$f_\rho(\rho)\in\Der A,$ 	it is enough to find \(f_\rho(t)\) such that	$f_\rho(\rho)a_i=0,$	$	1\le i\le m.$	Consider the vector space
		\[
		V=\sum_{k=0}^{n(\rho)}\mathbb{F}\, t^{p^k}.
		\]
		Then	$\dim_\mathbb{F} V=n(\rho)+1.$	For \(1\le i\le m\), define a linear mapping	$\varphi_i\colon V\to A,$ $\varphi_i(f(t))=f(\rho)a_i.$	By the definition of \(S_i\), we have	$\codim_V \Ker\varphi_i\le S_i.$	Consequently,
		\[
		\codim_V\bigcap_{i=1}^{m}\Ker\varphi_i
		\le
		S_1+\cdots+S_m
		=
		n(\rho).
		\]
		Since \(\dim_\mathbb{F} V=n(\rho)+1\), it follows that
		\[
		\bigcap_{i=1}^{m}\Ker\varphi_i\ne 0.
		\]
		Choose a nonzero polynomial
		\[
		f_\rho(t)\in \bigcap_{i=1}^{m}\Ker\varphi_i.
		\]
	
	Then \(f_\rho(\rho)a_i=0\) for \(1\le i\le m\). Hence,
	\(f_\rho(\rho)\) vanishes on the generators of \(A\).
	Since \(f_\rho(\rho)\) is a derivation, we obtain
	\(f_\rho(\rho)=0\).
	Multiplying by a nonzero scalar if necessary, we may assume
	that \(f_\rho(t)\) is monic.
	
	If \(\deg f_\rho(t)=p^j\) for some \(j<n(\rho)\), then
	the polynomial
	\[
	g_\rho(t)=f_\rho(t)^{p^{n(\rho)-j}}
	\]
	is monic of degree \(p^{n(\rho)}\). Moreover,
	\[
	g_\rho(\rho)
	=f_\rho(\rho)^{p^{n(\rho)-j}}=0.
	\]
	Thus, we may assume that \(f_\rho(t)\) is monic
	of degree \(p^{n(\rho)}\).
	\end{proof}
	
	We shall also use the following result.
	
	\begin{proposition}[{\cite{Zelmanov}, Theorem 1.1}]\label{prop:zelmanov}
	
	Let \(R\) be an associative \(\mathbb{F}\)-algebra generated by
	elements \(u_1,\ldots,u_n\), and let \(L\subseteq R\) be the Lie
	subalgebra generated by these elements. Suppose that
	\begin{enumerate}
		\item[(1)] \(L\) satisfies a polynomial identity;
		
		\item[(2)] for every commutator \(\rho\) in \(u_1,\ldots,u_n\),
		there exists an integer \(N_\rho\geq 1\) such that
		\[
		\rho^{N_\rho}=0.
		\]
	\end{enumerate}
		Then the algebra \(R\) is nilpotent.
	\end{proposition}
	
	\begin{remark} 	
		In fact, Theorem 1.1 of \cite{Zelmanov} asserts that if
		$L$ satisfies a polynomial identity and the adjoint
		operator $\ad(\rho)$ is nilpotent for every commutator
		$\rho$ in $u_1,\ldots,u_n$, then the Lie algebra $L$
		is finite-dimensional.
		
		If $\rho^n=0$ in the associative algebra $R$, then	$\ad(\rho)^{2n-1}=0.$	Proposition~\ref{prop:zelmanov} then follows from
		this theorem and the Poincar\'e--Birkhoff--Witt theorem.
					\end{remark}

	Recall that \(A\) is a commutative \(\mathbb{F}\)-algebra generated by elements $a_1,\ldots,a_m.$ Then the Lie \(\mathbb{F}\)-algebra \(\Der A\) is a homomorphic image of a Lie subalgebra of $\Der \mathbb{F}[x_1,\ldots,x_m],$ where \(\mathbb{F}[x_1,\ldots,x_m]\) denotes the polynomial algebra in \(m\) variables over \(\mathbb{F}\). Indeed, let  $A\cong  \mathbb{F}[x_1,\ldots,x_m]/J.$ Define $$ N(J)=\{d\in \Der \mathbb{F}[x_1,\ldots,x_m] \mid d(J) \subseteq J \}.$$ Then the natural map $N(J)\to \Der A$ is a surjective homomorphism.
	
	Let \(\mathbb{F}_p\) be the prime subfield of \(\mathbb{F}\). We have
	\[
	\Der \mathbb{F}[x_1,\ldots,x_m]
	=
	\left\{
	\sum_{i=1}^{m}
	f_i(x_1,\ldots,x_m)\frac{\partial}{\partial x_i}
	\ \middle|\
	f_i\in \mathbb{F}[x_1,\ldots,x_m],\ 1\leq i\leq m
	\right\}.
	\]
	Moreover, this Lie algebra is the \(\mathbb{F}\)-linear span of the Lie \(\mathbb{F}_p\)-algebra
	\[
	\Der \mathbb{F}_p[x_1,\ldots,x_m]
	=
	\left\{
	\sum_{i=1}^{m}
	f_i(x_1,\ldots,x_m)\frac{\partial}{\partial x_i}
	\ \middle|
	f_i\in \mathbb{F}_p[x_1,\ldots,x_m],\ 1\leq i\leq m
	\right\}.
	\]

	By \cite[Proposition~4.1]{BezushchakPetravchukZelmanov}, the Lie algebra	$\Der \mathbb{F}_p[x_1,\ldots,x_m]$ satisfies a multilinear identity \(P=0\) with coefficients \(\pm 1\).
	Since
\[
\Der \mathbb{F}[x_1,\ldots,x_m]
\cong
\mathbb{F}\otimes_{\mathbb{F}_p}
\Der \mathbb{F}_p[x_1,\ldots,x_m]
\]
	and \(\Der A\) is a homomorphic image of a subalgebra of
	\(\Der \mathbb{F}[x_1,\ldots,x_m]\),	it follows that both	$\Der \mathbb{F}[x_1,\ldots,x_m]$ and	$\Der A$	satisfy the identity \(P=0\).

	Let	$\mathbb{F}\langle x_1,\ldots,x_n\rangle$	be the free associative algebra without identity, and let \(I\) be the ideal of   $ \mathbb{F}\langle x_1,\ldots,x_n\rangle$ generated by all elements of the form
\[
P(\sigma_1,\ldots,\sigma_l)
\quad\text{and}\quad
\rho(x_1,\ldots,x_n)^{p^{n(\rho)}},
\]
where \(\sigma_i\) and \(\rho\) range over all
commutators in \(x_1,\ldots,x_n\). Since \(P\) is multilinear and every iterated commutator in \(x_1,\ldots,x_n\) is homogeneous, the ideal \(I\) is homogeneous.

	By Proposition~\ref{prop:zelmanov}, the algebra	$\mathbb{F}\langle x_1,\ldots,x_n\rangle/I$	is nilpotent. Hence, there exists an integer \(N\ge 1\) such that
	\[
	\mathbb{F}\langle x_1,\ldots,x_n\rangle^N\subseteq I.
	\]
	
	\section{Proof of the main theorem}
	
	\begin{proof}[Proof of Theorem~\ref{thm:main}]
		We show that the algebra \(U\subseteq End_{\mathbb{F}} A\) generated  by $d_{1},\ldots, d_{n}$ is finite-dimensional. It suffices to prove
		that \(U\) is spanned by products	$d_{i_1}\cdots d_{i_s}$ of length $		s<N.$
		
		Consider an arbitrary product	$	d_{i_1}\cdots d_{i_N}.$	Since
		\[
		\mathbb{F}\langle x_1,\ldots,x_n\rangle^N
		\subseteq I,
		\]
		the corresponding word \(x_{i_1}\cdots x_{i_N}\)
		belongs to \(I\) and can therefore be expressed as a
		linear combination of elements involving generators of \(I\):
		\[
		\begin{aligned}
			x_{i_1}\cdots x_{i_N}
			={}&
			\sum \alpha\,
			x_{j_1}\cdots x_{j_k}\,
			\rho^{p^{n(\rho)}} x_{\mu_1}\cdots x_{\mu_s}
				+
			\sum \beta\,
			x_{\nu_1}\cdots x_{\nu_r}\,
			p(\sigma_1,\ldots,\sigma_l)\,
			x_{w_1}\cdots x_{w_t},
		\end{aligned}
		\]
		where \(\alpha,\beta\in \mathbb{F}\).	The preceding equality holds in the free  associative algebra
		\(\mathbb{F}\langle x_1,\ldots,x_n\rangle\).
		
		Since the elements
		\[
		\rho(x_1,\ldots,x_n)^{p^{n(\rho)}},
		\qquad
		P(\sigma_1,\ldots,\sigma_l)
		\]
		are homogeneous, it follows that  each
		summand on the right-hand side has degree
		\[
		k+p^{n(\rho)}\deg(\rho)+s
		=
		r+\deg P(\sigma_1,\ldots,\sigma_l)+t
		=
		N.
		\]
	After substituting	$x_i\mapsto d_i,$ $ 1\leq i\leq n,$	we get 	\[
d_{i_1}\cdots d_{i_N}= \sum	\alpha\,
	d_{j_1}\cdots d_{j_k}
	\rho(d_1,\ldots,d_n)^{p^{n(\rho)}}
	d_{\mu_1}\cdots d_{\mu_s},
	\qquad
	\alpha\in\mathbb{F}.
	\]

	By Lemma~\ref{lem:polynomial},	$f_\rho(\rho)=0.$	Since \(f_\rho(t)\) is monic and has the form
	\[
	f_\rho(t)=
	\alpha_0t+\alpha_1t^p+\cdots+
	\alpha_{n(\rho)-1}t^{p^{n(\rho)-1}}
	+t^{p^{n(\rho)}},
	\]
	it follows that
	\[
	\rho^{p^{n(\rho)}}
	\]
	is a linear combination of
	\[
	\rho,\rho^p,\rho^{p^2},\ldots,
	\rho^{p^{n(\rho)-1}}.
	\]
	Consequently, every term of the form
	\[
	d_{j_1}\cdots d_{j_k}
	\rho(d_1,\ldots,d_n)^{p^{n(\rho)}}
	d_{\mu_1}\cdots d_{\mu_s}
	\]
	can be expressed as a linear combination of terms of the form
	\[
	d_{j_1}\cdots d_{j_k}
	\rho(d_1,\ldots,d_n)^{p^t}
	d_{\mu_1}\cdots d_{\mu_s},
	\qquad
	0\leq t<n(\rho).
	\]

	Thus, every product	$d_{i_1}\cdots d_{i_N}$	is a linear combination of products of length less than  \(N\). 	Since \(I\) is homogeneous, the same reduction applies to every product of length at least \(N\). Thus,  \(U\) is spanned by all products \(d_{i_1}\cdots d_{i_s}\) with \(1\le s<N\).
		There are only finitely many such products. Hence,	$\dim_{\mathbb{F}} U<\infty.$	This completes the proof.
	\end{proof}

\end{document}